\documentclass[onethmnum]{ws-rv9x6forp}
\usepackage{mathrsfs}
\makeindex
\newcommand{\Thm}[1]{Theorem~\ref{th:#1}}
\newcommand{\Prop}[1]{Proposition~\ref{prop:#1}}
\newcommand{\Lem}[1]{Lemma~\ref{lem:#1}}

\newcommand{\Rem}[1]{Remark~\ref{rem:#1}}

\newcommand{\Eq}[1]{\eqref{eq:#1}}
\makeatletter
\def\rom#1{\mbox{\leavevmode\skip@\lastskip\unskip\/\ifdim\skip@=\z@\else\hskip\skip@\fi{\rm{#1}}}}
\makeatother

\renewcommand{\b}{\beta}
\newcommand{\gm}{\gamma}\newcommand{\dl}{\delta}
\newcommand{\eps}{\varepsilon}
\newcommand{\lm}{\lambda}
\newcommand{\sg}{\sigma}
\newcommand{\ph}{\varphi}
\newcommand{\om}{\omega}

\newcommand{\Sg}{\Sigma}

\newcommand{\bfone}{\boldsymbol{1}}

\newcommand{\D}{\mathbb{D}}
\newcommand{\N}{\mathbb{N}}
\newcommand{\Q}{\mathbb{Q}}
\newcommand{\R}{\mathbb{R}}

\newcommand{\cB}{\mathcal{B}}
\newcommand{\cC}{\mathcal{C}}

\newcommand{\cE}{\mathcal{E}}
\newcommand{\cF}{\mathcal{F}}

\newcommand{\cM}{\mathcal{M}}

\newcommand{\cX}{\mathcal{X}}
\newcommand{\fH}{\mathscr{H}}
\newcommand{\fP}{\mathscr{P}}

\newcommand{\la}{\langle}\newcommand{\ra}{\rangle}
\newcommand{\wg}{\wedge}
\newcommand{\loc}{\mathrm{loc}}
\newcommand{\BVloc}{\dot{BV}_{\!\loc}}

\newcommand{\FCb}{\mathcal{F}C_b}

\newcommand{\Cp}{\mathop{\mathrm{Cap}}\nolimits}
\newcommand{\esssup}{\mathop{\text{\rm -ess\,sup}}}

\begin{document}\pagestyle{plain}
\allowdisplaybreaks[3]
\chapter*{Functions of locally bounded variation on Wiener spaces}

\author[M. Hino]{Masanori Hino}

\address{Graduate School of Engineering Science, Osaka University\\
{hino@sigmath.es.osaka-u.ac.jp}}

\begin{abstract}
We introduce the concept of functions of locally bounded variation on abstract Wiener spaces and study their properties.
Some nontrivial examples and applications to stochastic analysis are also discussed.

\end{abstract}
\markright{Functions of locally bounded variation on Wiener spaces} 
\body

\section{Introduction}
Functions of bounded variation (BV functions) on abstract Wiener spaces were first studied in \cite{Fu,FH} for their applications in stochastic analysis, but recently BV functions on infinite dimensional spaces have attracted attention for a variety of reasons.
In this paper, we newly introduce the space $\BVloc$ of functions of locally bounded variation (local BV functions) on the abstract Wiener space $(E,H,\mu)$.
Needless to say, there could be several ways of localizing the concept of bounded variation.
In this article, we adopt the ideas in the theory of (quasi-)regular Dirichlet forms, which are suitable for application to stochastic analysis.
Indeed, we show that a Dirichlet form of type
\[
\cE^\rho(f,g)=\frac12\int_E \la\nabla f,\nabla g\ra_H\, \rho\,d\mu
\]
associated with a nonnegative function $\rho$ in $\BVloc$ with some extra assumptions provides a diffusion process that has the Skorokhod representation (\Thm{decomp}). 
This result is regarded as a natural generalization of \cite[Theorem~4.2]{FH}, where $\rho$ is assumed to be a BV function instead. We also consider the classical Wiener space on $\R^d$ as $E$ and provide a sufficient condition for an open set $O$ of $\R^d$ so that the indicator function of the set of all paths staying in the closure $\overline O$ is a local BV function.
Accordingly, we can construct the (modified) reflecting Ornstein--Uhlenbeck process with the Skorokhod representation on the set of paths staying in $\overline{O}$ under a rather weak condition on $O$.
This is a study related to another paper~\cite{HU}, in which a sufficient condition was given for the above-mentioned indicator function to be a BV function on either pinned path spaces or one-sided pinned spaces.

The remainder of this paper is organized as follows.
In Section~2, we introduce the concept of local BV functions on a Wiener space and study their properties.
In Section~3, we provide a sufficient condition for a class of indicator functions to be local BV functions.

\section{The space $\BVloc$ on Wiener spaces}
Let $(E,H,\mu)$ be an abstract Wiener space. That is,
$E$ is a separable Banach space, $H$ is a separable Hilbert
space densely and continuously embedded in $E$, and $\mu$ is a
Gaussian measure on $E$ that satisfies
\[
  \int_E \exp\left(\sqrt{-1}l(z)\right)\mu(dz)=\exp\left(-|l|_H^2/2\right),
  \quad l\in E^*.
\]
Here, we regard the topological dual of $E$, $E^*$, as a dense subspace of $H$ by the natural identification 
$H^*\simeq H$. The inner product and the norm
of $H$ are denoted by $\la\cdot,\cdot\ra_H$ and $|\cdot|_H$, respectively.
We set
\[
\FCb^1=\left\{u\colon E\to\R\left|\,\begin{array}{ll} u(z)=f(l_1(z),\ldots,l_m(z)),\ l_1,\ldots,l_m\in E^*,\\ f\in C_b^1(\R^m) \mbox{ for some }m\in\N \end{array}\right.\right\}
\]
and define $\FCb^1(E^*)$ as  the set of all linear combinations of $H$-valued functions $u(\cdot)l$ on $E$ with $u\in \FCb^1$ and $l\in E^*\subset H$.
For $u\in\FCb^1$, $\nabla u$ denotes the $H$-derivative of $u$, which is an $H$-valued function on $E$ that is characterized by the identity
\[
  \la\nabla u(z),l\ra_H =\lim_{\eps\to0}(u(z+\eps l)-u(z))/\eps, \quad l\in E^*\subset H\subset E.
\]
In the same way, $\nabla u$ is defined for $u\in\FCb^1(E^*)$ as an $H\otimes H$-valued function on $E$.

Let $(X,\cX,\nu)$ be a measure space and let $Y$ be a separable Hilbert space with norm $|\cdot|_Y$. For $p\in[1,+\infty]$, we denote by $L^p(X\to Y;\nu)$ the space of all $Y$-valued $L^p$-functions on $(X,\cX,\nu)$ with the norm defined as
\[
 \|f\|_{L^p(\nu)}=
 \begin{cases}
 \displaystyle \left(\int_X |f(x)|_Y^p\,\nu(dx)\right)^{1/p},& p\in[1,+\infty),\\
 \displaystyle \nu\esssup_{x\in E}|f(x)|_Y, &p=+\infty.
 \end{cases}
\]
As usual, two functions that are equal a.e.\ are identified. 
If $X=E$ or $\nu=\mu$ or $Y=\R$, we often omit these symbols from the notation. In particular, $L^p$ and $\|\cdot\|_{p}$ represent $L^p(E;\mu)$ and $\|\cdot\|_{L^p(\mu)}$, respectively.
For $p\in[1,+\infty)$, the closure of $\FCb^1$ (resp.\ $\FCb^1(E^*)$) with respect to the norm $(\|\cdot\|_p^p+\|\nabla\cdot\|_p^p)^{1/p}$ will be denoted by $\D^{1,p}$ (resp.\ $\D^{1,p}(H)$).
The operator $\nabla$ extends to a continuous map from $\D^{1,p}$ to $L^p(E\to H)$.
Let $\nabla^*$ denote the adjoint operator of $\nabla$.
This is regarded as a bounded operator from $\D^{1,p}(H)$ to $L^p$ for $p\in(1,\infty)$.
We note that both $\nabla$ and $\nabla^*$ have the local property in the sense that, if $f\in\D^{1,2}$ (resp.\ $G\in \D^{1,2}(H)$) satisfies $f=0$ (resp.\ $G=0$) $\mu$-a.e.\ on some measurable set $A$, then $\nabla f=0$ (resp.\ $\nabla^* G=0$) $\mu$-a.e.\ on $A$.
See, e.g., \cite[Propositions~1.3.16 and 1.3.15]{N} for the proof.

Let $L(\log L)^{1/2}$ denote the set of all real-valued $\mu$-measurable functions $f$ on $E$ such that $f(0\vee\log|f|)^{1/2}\in L^1$.
For $\rho\in L(\log L)^{1/2}$, define
\begin{align*}
  V(\rho)&=\sup\left\{
  \int_E(\nabla^* G)\rho\,d\mu \;\vrule\; G\in\FCb^1(E^*),\ |G(z)|_H\le 1\mbox{ for every } z\in E\right\}\\ &(\le+\infty).
\end{align*}
 The function space $BV$ on $E$ is defined as
 \[
 BV=\{\rho\in L(\log L)^{1/2}\mid V(\rho)<\infty\}.
 \]
  A function in $BV$ is called a BV function or a function of bounded variation.
  We remark that the following inequality holds for $p\in(1,+\infty]$ and $\rho\in BV\cap L^p$.
  \begin{equation}\label{eq:bvp}
  \int_E (\nabla^* G)\rho\,d\mu
  \le V(\rho)\|G\|_\infty,
  \quad G\in \D^{1,q}(H),
  \end{equation}
  where $q$ is the conjugate exponent of $p$. 
  
For a function space $\cC$ on $E$ and a $\mu$-measurable subset $A$ of $E$, $\cC_A$ denotes the space of all functions in $\cC$ vanishing $\mu$-a.e.\ on $E\setminus A$, and $\cC_b$ denotes the set of all bounded functions in $\cC$. Moreover, $\cC_{A,b}$ denotes $\cC_A\cap \cC_b$.

For $\xi\in L^1$, $F^\xi$ denotes the support of the measure $|\xi|\cdot\mu$.
For $\xi\in L^1$ with $\xi\ge0$ $\mu$-a.e., we define a bilinear form $(\cE,\FCb^1)$ on $L^2(F^\xi,\xi\cdot\mu)$ by
\[
\cE^\xi(f,g)=\frac12\int_E \la \nabla f,\nabla g\ra_H\,\xi\, d\mu,
\quad f,g\in\FCb^1.
\]
The set of all $\xi$ such that $(\cE,\FCb^1)$ is closable on $L^2(F^\xi;\xi\cdot\mu)$ is denoted by $QR$.
For $\xi\in QR$, the closure of $(\cE,\FCb^1)$, denoted by $(\cE^\xi,\cF^\xi)$, is a quasi-regular Dirichlet form on $L^2(F^\xi;\xi\cdot\mu)$ (see, e.g., \cite{RS} and \cite[Theorem~2.1]{Fu} for the proof).
The space $\cF^\xi$ is regarded as a Hilbert space by using the inner product $(f,g)\mapsto\cE^\xi(f,g)+\int_{F^\xi} f g\,\xi\,d\mu$.
The associated capacity $\Cp^\xi$ is then defined as 
\[
\Cp^\xi(A)=\inf\left\{\cE^\xi(f,f)+\|f\|_{L^2(\xi\cdot\mu)}^2\;\vrule\;\begin{array}{l} f\in\cF^\xi\text{ and } f\ge1\ \xi\cdot\mu\text{-a.e.\ on}\\\text{some open set including $A$}\end{array}\right\}
\]
for $A\subset F^\xi$.
The concept ``$\cE^\xi$-quasi-everywhere'' ($\cE^\xi$-q.e.) is based on this capacity.
An increasing sequence $\{F_k\}_{k=1}^\infty$ of closed sets in $F^\xi$ is called an $\cE^\xi$-nest if $\lim_{k\to\infty}\Cp^\xi(F^\xi\setminus F_k)=0$.
It is known that $\{F_k\}_{k=1}^\infty$ is an $\cE^\xi$-nest if and only if $\bigcup_{k=1}^\infty \cF^\xi_{F_k}$ is dense in $\cF^\xi$.
In this situation, the set $S^\xi$ of all smooth measures is described as the totality of all positive Borel measures $\nu$ on $F^\xi$ such that $\nu$ charges no set of zero capacity $\Cp^\xi$ and there exists an $\cE^\xi$-nest $\{F_k\}_{k=1}^\infty$ such that $\nu(F_k)<\infty$ for all $k$.
A function $f$ on $E$ is called $\cE^\xi$-quasi-continuous if there exists an $\cE^\xi$-nest $\{F_k\}_{k=1}^\infty$ such that $f|_{F_k}$ is continuous on $F_k$ for every $k$.
Any function $f\in\cF^\xi$ has an $\cE^\xi$-quasi-continuous modification $\tilde f$. 
When $\xi\equiv1$, we write $\cE$, $\cF$, and $\Cp$ instead of $\cE^\xi$, $\cF^\xi$, and $\Cp^\xi$, respectively.
We note that $\cF=\D^{1,2}$.
For a subset $A$ of $E$, $e_A$ denotes the equilibrium potential of $A$ with respect to $(\cE,\cF)$: that is, $e_A$ attains the infimum of $\{\cE(f,f)+\|f\|_2^2\mid f\in \D^{1,2}\mbox{ and } \tilde f\ge1 \mbox{ $\cE$-q.e.\ on }A\}$, which is equal to $\Cp(A)$ (cf.~\cite[Theorem~2.1.5]{FOT}).
Henceforth, we will always assume that $e_A$ is $\cE$-quasi-continuous by itself.
Any $G\in\D^{1,2}(H)$ also has an $\cE$-quasi-continuous modification $\tilde G$ (cf.\ \cite[Chapter~VII, Theorem~1.3.1]{BH}).

One of the fundamental properties of BV functions is given by the following theorem.
\begin{theorem}[cf.~{\cite[Theorem~3.9]{FH}}]\label{th:BV}
  For $\rho\in BV$, there exists a positive finite measure $\nu$ on $E$ and an $H$-valued Borel function $\sg$ on $E$ such that $|\sg|_H=1$ $\nu$-a.e., and for every $G\in \FCb^1(E^*)$,
 \begin{equation}\label{eq:measure}
 \int_E(\nabla^* G)\rho\,d\mu=\int_E \la G,\sg\ra_H \,d\nu.
 \end{equation}
 The measure $\nu$ belongs to $S^{|\rho|+1}$.
 If $\rho\in QR$ in addition, then $\nu|_{E\setminus F^\rho}=0$ and $\nu|_{F^\rho}\in S^\rho$.
 Also, $\nu$ and $\sg$ are uniquely determined in the following sense: if $\nu'$ and $\sg'$ are another pair satisfying \Eq{measure} for all $G\in \FCb^1(E^*)$, then $\nu=\nu'$ and $\sg=\sg'$ $\nu$-a.e.
 \end{theorem}
 
We now introduce some localized function spaces.
\begin{definition}\label{def:loc}
Let $\rho$ be a real-valued $\mu$-measurable function on $E$.
\begin{romanlist}[(ii)]
\item For $p\in[1,+\infty]$, we say $\rho\in \dot L^p_\loc$ if there exists an $\cE$-nest $\{F_k\}_{k=1}^\infty$ such that $\rho\cdot \bfone_{F_k}\in L^p$ for every $k\in\N$.
\item We say $\rho\in \BVloc$ if there exists an $\cE$-nest $\{F_k\}_{k=1}^\infty$ and $\{\rho_k\}_{k=1}^\infty\subset BV$ such that $\rho=\rho_k$ $\mu$-a.e.\ on $F_k$ for every $k\in\N$.
\end{romanlist}
\end{definition}
The following are some properties of $\BVloc$.
\begin{theorem}\label{th:1}
For $\rho\in \dot L^2_\loc$, the implications $\rom{(a)}\Leftrightarrow\rom{(b)}\Leftarrow\rom{(c)}$ hold.
\begin{alphlist}[\rom{(c)}]
\item[\rom{(a)}] $\rho\in \dot{BV}_\loc$.
\item[\rom{(b)}] There exists an $\cE$-nest $\{F_k\}_{k=1}^\infty$ and positive numbers $\{c_k\}_{k=1}^\infty$ such that for every $k\in\N$, $\rho\cdot\bfone_{F_k}\in L^2$ and
\[
  \int_E \rho\nabla^* G\,d\mu\le c_k \|G\|_\infty,\quad
  G\in \D^{1,2}(H)_{F_k,b}.
\]
\item[\rom{(c)}] There exists an $\cE$-nest $\{F_k\}_{k=1}^\infty$ and $\{f_m\}_{m=1}^\infty\subset \D^{1,1}$ such that for every $k\in \N$, $(f_m-\rho)\cdot \bfone_{F_k}$ converges to $0$ in $L^1$ as $m\to\infty$ and $\|(\nabla f_m)\cdot \bfone_{F_k}\|_1$ is bounded in $m$.
\end{alphlist}
\end{theorem}
Unlike the case of finite-dimensional spaces, we must take care with the localized test functions on $E$. Indeed, except for the zero function there are no functions $f$ in $\FCb^1$ such that $f$ vanishes outside a bounded set.
This is why we introduce $\D^{1,2}(H)_{F_k,b}$ in the theorem above.

For the proof of \Thm{1}, we introduce some lemmas.
\begin{lemma}\label{lem:nest}
Let $\{F_k\}_{k=1}^\infty$ be an $\cE$-nest. 
Then, there exists another $\cE$-nest $\{F'_k\}_{k=1}^\infty$ and functions $\{\ph_k\}_{k=1}^\infty\subset \bigcup_{n=1}^\infty\D^{1,2}_{F_n}$ such that for any $k$, $F'_k\subset F_l$ for some $l\ge k$, $\ph_k$ is $\cE$-quasi-continuous, $0\le \ph_k\le1$ on $E$, $\ph_k=1$ on $F'_k$, and $\lim_{k\to\infty}\cE(\ph_k,\ph_k)=0$.
\end{lemma}
\begin{proof}
We denote $E\setminus F_k$ by $F_k^c$, etc.
Take an increasing sequence $\{n(k)\}_{k=1}^\infty$ of natural numbers such that $\Cp(F_{n(k)}^c)<2^{-k}$ for every $k$.
There exists an $\cE$-nest $\{\hat F_l\}_{l=1}^\infty$ such that $e_{F_{n(k)}^c}$ is continuous on each $\hat F_l$ for every $k$ (cf.\ \cite[Theorem~2.1.2]{FOT}).
For $k\in\N$, let $\ph_k=1\wg 2(1- e_{F_{n(k)}^c})$ and $\check F_k=\{\ph_k=1\}$.
Since $\check F_k^c=\{e_{F_{n(k)}^c}>1/2\}$, 
\[
\Cp(\check F_k^c)\le \cE(2e_{F_{n(k)}^c},2e_{F_{n(k)}^c})+\|2e_{F_{n(k)}^c}\|_2^2
=4\Cp(F_{n(k)}^c)<2^{-k+2}.
\]
Define $F'_k=\bigcap_{l=k}^\infty (\check F_l\cap \hat F_l\cap F_{n(l)})$.
Then, $\{F'_k\}_{k=1}^\infty$ is a nondecreasing sequence of closed sets and
\begin{align*}
\Cp((F'_k)^c)
&=\Cp\left(\left(\bigcup_{l=k}^\infty \check F_l^c\right)\cup \hat F_k^c\cup F_{n(k)}^c\right)\\
&\le \sum_{l=k}^\infty\Cp(\check F_l^c)+\Cp(\hat F_k^c)+  \Cp(F_{n(k)}^c)\\
&\le 2^{-k+3}+\Cp(\hat F_k^c)+\Cp(F_{n(k)}^c),
\end{align*}
which converges to $0$ as $k\to\infty$.
Furthermore, $\cE(\ph_k,\ph_k)\le 4\cE(e_{F_{n(k)}^c},e_{F_{n(k)}^c})<2^{-k+2}\to0$ as $k\to\infty$.
\end{proof}
\begin{lemma}\label{lem:bv}
Let $p\in(1,+\infty]$ and $q$ denote the conjugate exponent of $p$.\begin{romanlist}[\rom{(ii)}]
\item[\rom{(i)}] 
Let $f\in BV\cap L^p$ and $g\in \D^{1,q}_b$. Then, $fg\in BV$ and $V(fg)\le V(f)\|g\|_\infty+\|f\nabla g\|_1$.
\item[\rom{(ii)}] Let $f\in BV$ and $g\in \D^{1,q}_{F,b}$ for some $\mu$-measurable set $F\subset E$.
If $f\cdot\bfone_F\in L^p$, then $fg\in BV\cap L^p$ and $V(fg)\le V(f)\|g\|_\infty+\|f\nabla g\|_1$.
\end{romanlist}
\end{lemma}
\begin{proof}
(i) Let $\{T_t\}_{t>0}$ be the Ornstein--Uhlenbeck semigroup on $E$.
For $t>0$, $T_t f\in \D^{1,p}$, thus $(T_t f)g\in\D^{1,1}$ and
\begin{align*}
  \|\nabla((T_tf)g)\|_1
  &\le \|\nabla (T_t f)\|_1\|g\|_\infty+\|(T_t f)\nabla g\|_1\\
  &\le e^{-t}V(f)\|g\|_\infty+\|(T_t f)\nabla g\|_1,
\end{align*}
from, e.g.,~\cite[Proposition~3.6]{FH}.
Since $\lim_{t\to0}(T_tf)g= fg$ in $L^1$ and $\lim_{t\to0}(T_t f)\nabla g=f\nabla g$ in $L^1(H)$, Theorem~3.7 of \cite{FH} completes the proof.

(ii) For $n\in \N$, let $f_n=(-n)\vee (f\wg n)$. 
From \cite[Corollary~3.8]{FH}, $f_n\in BV$ and $V(f_n)\le V(f)$.
By the assumption and (i), $f_n g\in BV\cap L^p$,
$\|f_n g\|_p\le \|f g\|_p$, and 
\[
V(f_n g)\le V(f_n)\|g\|_\infty+\|f_n\nabla g\|_1\le V(f)\|g\|_\infty+\|f \nabla g\|_1,
\]
which is bounded for each $n$.
Since $\lim_{n\to\infty}f_n g= f g$ in $L^1$, the claim follows.
\end{proof}
\begin{lemma}\label{lem:nest2}
For $\rho\in \BVloc\cap\dot L^2$, there exists an $\cE$-nest $\{F_k\}_{k=1}^\infty$ and functions $\{\rho_k\}_{k=1}^\infty$ in $BV\cap L^2$ such that $\rho=\rho_k$ on $F_k$ for every $k\in\N$.
\end{lemma}
\begin{proof}
There exist an $\cE$-nest $\{F_k\}_{k=1}^\infty$ and $\{\rho_k\}_{k=1}^\infty\subset BV$ such that for every $k$, $\rho\cdot\bfone_{F_k}\in L^2$ and $\rho=\rho_k$ on $F_k$.
Take the $\cE$-nest $\{F'_k\}_{k=1}^\infty$ and the functions $\{\ph_k\}_{k=1}^\infty$ in \Lem{nest}.
From \Lem{bv}, the assertion is true by taking $\{F'_k\}_{k=1}^\infty$ and $\{\rho \ph_k\}_{k=1}^\infty$ as $\{F_k\}_{k=1}^\infty$ and $\{\rho_k\}_{k=1}^\infty$, respectively.
\end{proof}

\begin{proof}[Proof of \Thm{1}]
(a)$\Rightarrow$(b): 
Take the $\cE$-nest $\{F_k\}_{k=1}^\infty$ and $\{\rho_k\}_{k=1}^\infty\subset BV\cap L^2$ as in \Lem{nest2}.
Let $k\in\N$ and take $G\in \D^{1,2}(H)_{F_k,b}$.
Since $\nabla^* G=0$ $\mu$-a.e.\ on $E\setminus F_k$,
\[
\int_E \rho\nabla^* G\,d\mu
=\int_E \rho_k \nabla^* G\,d\mu
\le V(\rho_k)\|G\|_\infty,
\]
where the last inequality follows from \Eq{bvp}.
Thus, (b) holds.

(b)$\Rightarrow$(a):
Take the $\cE$-nest $\{F'_k\}_{k=1}^\infty$ and the functions $\{\ph_k\}_{k=1}^\infty$ as in \Lem{nest}.
For $k\in\N$, take $l\in\N$ such that $\ph_k\in\D^{1,2}_{F_l}$ and let $\rho_k=\rho\ph_k$. Then, $\rho_k\in L^2 $ and $\rho_k=\rho$ on $F'_k$.
Moreover, for any $G\in \FCb^1(E^*)$,
\begin{align*}
\int_E \rho_k\nabla^* G\,d\mu
&=\int_E \rho\ph_k\nabla^* G\,d\mu\\
&=\int_E \rho\{\nabla^*(\ph_k G)+\la \nabla \ph_k,G\ra_H\}\,d\mu\\
&\le c_l\|\ph_k G\|_\infty +\|\rho\cdot \bfone_{F_l}\|_2\|\nabla \ph_k\|_2\|G\|_\infty\\
&\le (c_l+\|\rho\cdot \bfone_{F_l}\|_2\|\nabla \ph_k\|_2)\|G\|_\infty.
\end{align*}
Therefore, $\rho_k\in BV$.

(c)$\Rightarrow$(b):
We may assume that $\rho\cdot\bfone_{F_k}\in L^2$ for every $k$.
For $M>0$, define $\Phi_M(t)=(-M)\vee(t\wg M)$ for $t\in\R$.
Then, for $k\in \N$ and $G\in \D^{1,2}(H)_{F_k,b}$,
\begin{align*}
\int_E \Phi_M(f_m)\nabla^*G\,d\mu
&=\int_E \la \nabla(\Phi_M(f_m)),G\ra_H\,d\mu\\
&\le \|(\nabla f_m)\cdot \bfone_{F_k}\|_1\|G\|_\infty
\le c_k\|G\|_\infty,
\end{align*}
where $c_k:=\sup_{m\in\N}\|(\nabla f_m)\cdot \bfone_{F_k}\|_1<\infty$.
Since $\nabla^*G\in L^2$, 
\[
\lim_{m\to\infty}\int_E \Phi_M(f_m)\nabla^*G\,d\mu=\int_E \Phi_M(\rho)\nabla^*G\,d\mu.
\]
Since $\lim_{M\to\infty}(\Phi_M(\rho)-\rho)\cdot \bfone_{F_k}=0$ in $L^2$, \[
\lim_{M\to\infty}\int_E \Phi_M(\rho)\nabla^*G\,d\mu=\int_E \rho\nabla^*G\,d\mu.
\]
Therefore, we have $\int_E \rho\nabla^*G\,d\mu\le c_k\|G\|_\infty$.
\end{proof}
The following are some basic lemmas that are used later.
\begin{lemma}\label{lem:easy}
The following claims hold.
\begin{romanlist}[\rom{(ii)}]
\item[\rom{(i)}] For $\rho\in L^\infty$, $\cF^{|\rho|+1}=\D^{1,2}$ and their norms are equivalent. 
\item[\rom{(ii)}]Let $\rho\in QR\cap L^\infty$ and $\{F_k\}_{k=1}^\infty$ be an $\cE$-nest. Then, $\D^{1,2}|_{F^\rho}$ is continuously embedded in $\cF^{\rho}$ and there exists $c>0$ such that $\Cp^\rho(A)\le c\Cp(A)$ for any $A\subset F^\rho$.
In particular, $\{F_k\cap F^\rho\}_{k=1}^\infty$ is an $\cE^\rho$-nest.
Moreover, $\bigcup_{k=1}^\infty \D^{1,2}_{F_k,b}|_{F^\rho}$ is dense in $\cF^\rho$.
\end{romanlist}
\end{lemma}
\begin{proof}
We prove only the last claim of (ii).
It suffices to prove that for any function $f\in\FCb^1$, $f|_{F^\rho}$ can be approximated by elements of $\bigcup_{k=1}^\infty \D^{1,2}_{F_k,b}|_{F^\rho}$ in $\cF^\rho$. Since $\bigcup_{k=1}^\infty \D^{1,2}_{F_k,b}$ is dense in $\D^{1,2}$, there exists $\{f_n\}_{n=1}^\infty$ in $\bigcup_{k=1}^\infty \D^{1,2}_{F_k,b}$ such that $f_n$ converges to $f$ in $\D^{1,2}$.
Then, $f_n|_{F^\rho}$ converges to $f|_{F^\rho}$ in $\cF^\rho$.
\end{proof}
\begin{lemma}\label{lem:uniqueness}
Let $\{F_k\}_{k=1}^\infty$ be an $\cE$-nest.
Let $\nu$ and $\nu'$ be $\cE$-smooth measures on $E$ such that $\nu(F_k)<\infty$ and $\nu'(F_k)<\infty$ for every $k\in\N$.
Let $\sg$ and $\sg'$ be $H$-valued Borel functions on $E$ such that $|\sg|_H=1$ $\nu$-a.e.\ and $|\sg'|_H=1$ $\nu'$-a.e.
If $\int_E \la \tilde G,\sg\ra_H\,d\nu=\int_E \la \tilde G,\sg'\ra_H\,d\nu'$ for every $G\in \bigcup_{k=1}^\infty \D^{1,2}(H)_{F_k,b}$, then $\nu=\nu'$ and $\sg=\sg'$ $\nu$-a.e.
\end{lemma}
\begin{proof}
 Let $\xi=\nu+\nu'$ and $\gm=\sg\frac{d\nu}{d\xi}-\sg'\frac{d\nu'}{d\xi}$.
 Then, $\int_E \la \tilde G,\gm\ra_H\,d\xi=0$ for every $G\in\bigcup_{k=1}^\infty \D^{1,2}_{F_k,b}$.
Taking a uniformly bounded sequence $\{G_n\}_{n=1}^\infty$ from $\bigcup_{k=1}^\infty \D^{1,2}_{F_k,b}$ such that $\la \tilde G_n,\gm\ra_H\to|\gm|_H$ $\xi$-a.e.\ as $n\to\infty$, we obtain $\gm=0$ $\xi$-a.e.
Therefore, $|\sg|_H\frac{d\nu}{d\xi}=|\sg'|_H\frac{d\nu'}{d\xi}$ $\xi$-a.e.
 Since $|\sg|_H=1$ $\nu$-a.e., $|\sg|_H\frac{d\nu}{d\xi}=\frac{d\nu}{d\xi}$ $\xi$-a.e.
  Similarly, $|\sg'|_H\frac{d\nu'}{d\xi}=\frac{d\nu'}{d\xi}$ $\xi$-a.e.
   Then, $\frac{d\nu}{d\xi}=\frac{d\nu'}{d\xi}$ $\xi$-a.e., which implies $\nu=\nu'$.
  The identity $\sg=\sg'$ $\nu$-a.e.\ follows from $\gm=0$ $\xi$-a.e.\ and $\nu=\nu'$.
\end{proof}
\begin{theorem}\label{th:2}
For $\rho\in \BVloc\cap \dot L^2$, there exist an $H$-valued Borel function $\sg$ on $E$, an $\cE$-smooth measure $\nu$ on $E$, and an $\cE$-nest $\{F_k\}_{k=1}^\infty$ such that $\rho\cdot\bfone_{F_k}\in L^2$ and $\nu(F_k)<\infty$ for every $k\in\N$, $|\sg|_H=1$ $\nu$-a.e., and
\begin{equation}\label{eq:ibp}
  \int_{E} \rho \nabla^* G\,d\mu
  =\int_E \la \tilde G,\sg\ra_H\,d\nu
  \quad\mbox{for every }G\in\bigcup_{k=1}^\infty \D^{1,2}(H)_{F_k,b},
\end{equation}
where $\tilde G$ denotes an $\cE$-quasi-continuous modification of $G$.
The pair $\nu$ and $\sg$ is uniquely determined in the following sense: if another pair $\nu'$ and $\sg'$ with some $\cE$-nest satisfies the conditions above, then $\nu=\nu'$ and $\sg=\sg'$ $\nu$-a.e.

If $\rho\in QR\cap L^\infty$ in addition, then $\nu|_{E\setminus F^\rho}=0$ and $\nu|_{F^\rho}\in S^\rho$.
\end{theorem}
\begin{proof}
Take an $\cE$-nest $\{F_k\}_{k=1}^\infty$ and $\{\rho_k\}_{k=1}^\infty\subset BV\cap L^2$ as in \Lem{nest2}.
From \Thm{BV}, for each $k$, there exist an $H$-valued Borel function $\sg_k$ on $E$ and an $\cE$-smooth finite measure $\nu_k$ on $E$ such that $|\sg_k|_H=1$ $\nu_k$-a.e.\ and
\begin{equation}\label{eq:id}
  \int_E \rho_k \nabla^* G\,d\mu
  =\int_E \la G,\sg_k\ra_H\,d\nu_k
\end{equation}
for every $G\in\FCb^1(E^*)$.
By approximation, \Eq{id} holds for $G\in \D^{1,2}(H)_{b}$ with $\la G,\sg_k\ra_H$ replaced by $\la \tilde G,\sg_k\ra_H$.

Take an $\cE$-nest $\{F'_k\}_{k=1}^\infty$ and functions $\{\ph_k\}_{k=1}^\infty$ as in \Lem{nest}. For $k\in \N$, let $n(k)$ denote a number such that $\ph_k\in \D^{1,2}_{F_{n(k)}}$. 
We may assume that $\{n(k)\}_{k=1}^\infty$ is an increasing sequence.
Then, for $l\ge k$,
\begin{equation}\label{eq:consistency}
\ph_k\cdot\nu_{n(k)}=\ph_k\cdot\nu_{n(l)}\text{ and }
\sg_{n(k)}=\sg_{n(l)} \ (\ph_k\cdot\nu_{n(k)})\text{-a.e.}
\end{equation}
Indeed, for any $G\in \D^{1,2}(H)_b$, we have $\ph_kG\in\D^{1,2}(H)_{F_{n(k)},b}$ and
\begin{align*}
\int_E \ph_k\la \tilde G,\sg_{n(k)}\ra_H\,d\nu_{n(k)}
&=\int_E \rho_{n(k)}\nabla^*(\ph_k G)\,d\mu\\
&=\int_E \rho_{n(l)}\nabla^*(\ph_k G)\,d\mu\\
&=\int_E \ph_k\la \tilde G,\sg_{n(l)}\ra_H\,d\nu_{n(l)}.
\end{align*}
Thus, \Eq{consistency} follows from \Lem{uniqueness}.
Therefore, we can define $\sg$ and $\nu$ so that $\sg=\sg_{n(k)}$ on $F'_k$ and $\nu|_{F'_k}=\nu_{n(k)}|_{F'_k}$ for $k\in\N$, and $\nu(E\setminus\bigcup_{k=1}^\infty F'_k)=0$.
Then, the conditions described in the theorem are satisfied with $\{F_k\}_{k=1}^\infty$ replaced by $\{F'_k\}_{k=1}^\infty$.

Suppose that another $\nu'$, $\sg'$, and an $\cE$-nest $\{F''_k\}_{k=1}^\infty$ satisfy the required conditions. 
Applying \Lem{uniqueness} to the $\cE$-nest $\{F_k\cap F''_k\}_{k=1}^\infty$, we obtain that $\nu=\nu'$ and $\sg=\sg'$ $\nu$-a.e. 

  Lastly, suppose $\rho\in QR\cap L^\infty$ in addition.
  Let $l\in E^*$ and $g\in \D^{1,2}_{{F_k},b}$ for some $k\in\N$.
  Letting $G=g(\cdot)l$ in \Eq{ibp}, we obtain
  \begin{equation}\label{eq:Frho}
  \int_{F^\rho} (\partial_l g+g\cdot l(\cdot))\rho\,d\mu=\int_E \tilde g\la l,\sg\ra_H\,d\nu.
  \end{equation}
  Denote the left-hand side of \Eq{Frho} by $I(g)$. Then, $I$ provides a bounded functional on $\cF^\rho$ as well as on $\D^{1,2}$. From the observation that $I(g)$ does not change if $g$ is replaced by $(-M)\vee (g\wg M)$ with $M=\mu\esssup |g\cdot \bfone_{F_k}|$, the inequality $I(g)\le \|g|_{F^\rho}\|_{L^\infty(\rho\cdot\mu)}|l|_H\nu(F_k)$ holds.
 Then, from \Lem{easy} and \cite[Theorem~2.18]{Hi}, there exists a unique $\nu_l\in S^\rho$ such that $I(g)=\int_{F^\rho} \tilde g\,d\nu_l$ for all $g\in \bigcup_{k=1}^\infty \D^{1,2}_{F_k,b}$. From the uniqueness of the integral representation of $I(g)$ as a functional on $\D^{1,2}$, we conclude that $(\la l,\sg\ra_H\cdot\nu)|_{E\setminus F^\rho}=0$ and $(\la l,\sg\ra_H\cdot\nu)|_{F^\rho}=\nu_l\in S^\rho$.
  Because $l$ is arbitrary, we deduce that $\nu|_{E\setminus F^\rho}=0$ and $\nu|_{F^\rho}\in S^\rho$.
\end{proof}
We will denote $\nu$ and $\sg$ in the theorem above by $\|D\rho\|$ and $\sg_\rho$, respectively.
\begin{theorem}[Skorokhod representation]\label{th:decomp}
  Let $\rho\in \BVloc\cap QR\cap L^\infty$. Then the sample 
  path of the diffusion process $\mathbf{M}^\rho=(X_t,\cM_t,P_z)$ associated with $(\cE^\rho,\cF^\rho)$
  admits the following expression as a sum of three $E$-valued continuous additive functionals.
  \begin{align}\label{eq:decomp}
    &X_t(\om)-X_0(\om)\notag\\
    &=W_t(\om)-\frac12\int_0^t X_s(\om)\,ds
    +\frac12\int_0^t\sg_\rho(X_s(\om))\,dA_s^{\|D\rho\|}(\om),
    \quad t\ge0.
 \end{align}
  Here, $A^{\|D\rho\|}$ is a real-valued positive continuous additive functional associated with $\|D\rho\|$ via the Revuz correspondence.
  Moreover, for $\cE^\rho$-q.e.~$z\in F^\rho$, $\{W_t\}_{t\ge0}$ is the $\{\cM_t\}$-Brownian motion 
  on $E$ under $P_z$.
\end{theorem}
\begin{proof}
The proof is provided along the same lines as those of \cite[Theorem~4.2]{FH} and \cite[Theorem~3.2]{Fu} with the use of \cite[Theorem~6.1]{Fu99} instead of \cite[Theorem~6.2]{Fu99} (or \cite[Theorem~2.2]{Fu}).
  We define $\{W_t\}_{t\ge0}$ so that \Eq{decomp} holds. Then, $\{W_t\}_{t\ge0}$ is an $E$-valued continuous additive functional with the same defining and exceptional sets as $A^{\|D\rho\|}$.
  Take the $\cE$-nest $\{F_k\}_{k=1}^\infty$ in \Thm{2}.
  Let $l\in E^*$ and consider the identity
 \begin{align*}
    l(X_t)-l(X_0)
    =l(W_t)-\frac12\int_0^t l(X_s)\,ds
    +\frac12\int_0^t\la l,\sg_\rho(X_s)\ra_H\,dA_s^{\|D\rho\|}.
 \end{align*} 
  We note that $l(\cdot)|_{F^\rho}\in \cF^\rho$.
  From \Eq{ibp}, the identity
  \[
  \cE^\rho(l(\cdot),g)=\frac12\int_{F^\rho} gl(\cdot)\rho\,d\mu-\frac12\int_{F^\rho}\tilde g\la l,\sg_\rho\ra_H\,d\|D\rho\|
  \]
  holds for $g\in \bigcup_{k=1}^\infty\D^{1,2}_{F_k,b}$.
  Take a nest $\{F'_k\}_{k=1}^\infty$ and functions $\{\ph_k\}_{k=1}^\infty$ as in \Lem{nest}.
  For any $g\in \cF^\rho_{F'_k,b}$ with $k\in\N$, there exist $\{g_n\}\subset \bigcup_{m=1}^\infty\D^{1,2}_{F_m,b}$ such that $\{g_n\}$ are uniformly bounded and $g_n|_{F^\rho}$ converges to $g$ in $\cF^\rho$ as $n\to\infty$.
  Since $\cE(g_n\ph_k,g_n\ph_k)^{1/2}\le \cE(g_n,g_n)^{1/2}\|\ph_k\|_\infty+\cE(\ph_k,\ph_k)^{1/2}\|g_n\|_\infty $, $\|g_n\ph_k\|_2\le \|g_n\|_\infty\|\ph_k\|_2$, and $(g_n\ph_k)|_{F^\rho}\to g$ in $L^2(\rho\cdot\mu)$, the Ces\`aro means of a certain subsequence of $\{(g_n\ph_k)|_{F^\rho}\}_{n=1}^\infty$, which are all elements of $\D^{1,2}_{F_l,b}|_{F^\rho}$ for some $l$, converges to $g$ in $\cF^\rho$.
  A further suitable subsequence of their $\cE^\rho$-quasi-continuous modifications converges $\cE^\rho$-q.e.\ from \cite[Theorem~2.1.4]{FOT}.
  Therefore, the above identity holds for $g\in\bigcup_{k=1}^\infty \cF^\rho_{F'_k,b}$, where $\tilde g$ is interpreted as an $\cE^\rho$-quasi-continuous modification of $g$.
 From \cite[Theorem~6.1]{Fu99}, the Fukushima decomposition of $l(X_t)-l(X_0)$ is given by the sum of $M_t^l:=l(W_t)$ and $N_t^l:=-\frac12\int_0^t l(X_s)\,ds
    +\frac12\int_0^t\la l,\sg_\rho(X_s)\ra_H\,dA_s^{\|D\rho\|}$.
  Moreover, the quadratic variation of $\{M_t^l\}$ is equal to $\{t|l|_H^2\}$.
  From \cite[Lemma~4.1]{FH}, we conclude that $\{W_t\}_{t\ge0}$ is the $E$-valued Brownian motion, which completes the proof.  
\end{proof}
\section{Indicator functions in $\BVloc$}
In this section, we provide some nontrivial examples of functions in $\BVloc$.
Let $d\in\N$ and $T>0$.
We consider the classical $d$-dimensional Wiener space as $(E,H,\mu)$: that is,
\begin{align*}
E&=\{w\in C([0,T]\to\R^d)\mid w(0)=0\},\\
H&=\left\{h\in E\;\vrule\;
\text{$h$ is absolutely continuous and } \int_0^T |\dot h(s)|_{\R^d}^2\,ds<\infty\right\},
\end{align*}
and $\mu$ is the Wiener measure on $E$.
For a subset $A$ of $\R^d$, $\overline{A}$ (resp.\ $\partial A$) denotes the closure (resp.\ boundary) of $A$, and $A^c$ denotes $\R^d\setminus A$.
Define some subsets of $E$ as follows:
\begin{align*}
\Xi_A&=\{w\in E\mid w(t)\in A\mbox{ for all } t\in[0,T]\},\\
\Theta_A&=\{w\in E\mid w(t)\in A\mbox{ for some } t\in[0,T]\},\\
\partial \Xi_A&=\left\{w\in E\;\vrule\;\parbox[l]{0.42\hsize}{$w(t)\in \overline{A}$ for all $t\in[0,T]$ and $w(s)\in \partial{A}$ for some $s\in[0,T]$}\right\},\\
\partial' \Xi_A&=\left\{w\in E\;\vrule\;\parbox[l]{0.53\hsize}{$w(t)\in \overline{A}$ for all $t\in[0,T]$ and $w(s)\in \partial{A}$ for some \emph{unique} $s\in[0,T]$}\right\}.
\end{align*}
We note that $\partial \Xi_A$ is the topological boundary of $\Xi_A$ in $E$ with the uniform topology.
A sufficient condition for $\bfone_{\Xi_A}$ to belong to $BV$ was given in \cite{U,HU}.
Here, we provide a sufficient condition for $\bfone_{\Xi_A}$ to belong to $\BVloc$.

For $x\in \R^d$ and $r\ge0$, we write $B(x,r)$ and $\overline B(x,r)$ for $\{z\in\R^d\mid |z-x|_{\R^d}<r\}$ and $\{z\in\R^d\mid |z-x|_{\R^d}\le r\}$, respectively.
Let $O$ be a proper open subset of $\R^d$ such that $0\in O$.
For $y\in\partial O$, we define $\dl(y)\in[0,+\infty]$ as 
\[
\dl(y)=\sup\{r\ge0\mid \mbox{there exists $z\in O^c$ such that }\overline O\cap \overline B(z,r)=\{y\}\}.
\]
From \cite[Theorem~5.1]{HU}, $\bfone_{\Xi_{\overline{O}}}\in BV$ if the uniform exterior ball condition 
\begin{equation}\label{eq:uebc}
\inf_{y\in\partial O}\dl(y)>0
\end{equation}
holds.
In order to describe a weaker condition which ensures that $\bfone_{\Xi_{\overline{O}}}\in \BVloc$, we introduce the Riesz (and logarithmic) capacities on $\R^d$.
For a Borel set $A$ of $\R^d$, let $\fP(A)$ denote the set of all Borel probability measures on $A$.
For $\b\ge0$, the $\b$-capacity $\Cp_\b(A)$ of $A$ is defined as
\[
{\Cp_\b(A)}=
\left(\inf_{\lm\in\fP(A)}\iint_{A\times A}g_\b(|x-y|_{\R^d})\,\lm(dx)\lm(dy)\right)^{-1},
\]
where $g_\b(t)=t^{-\b}$ for $\b>0$ and $g_0(t)=\log(t^{-1}\vee e)$.
For $\b<0$, we define $\Cp_\b(A)=1$ for $A\ne\emptyset$ and $\Cp_\b(\emptyset)=0$. These are Choquet capacities.

For $\b\ge0$, let $\fH^\b(A)$ denote the $\b$-dimensional Hausdorff measure of $A\subset \R^d$.
It is known that, if a Borel set $A$ of $\R^d$ satisfies $\fH^\b(A)<\infty$, then $\Cp_\b(A)=0$ (see, e.g., \cite[\S1, Theorem~1]{C}).

For $r>0$, a closed subset $\Sigma_r$ of $\R^d$ is defined as
\begin{equation}\label{eq:Sigma}
\Sigma_r=\overline{\{y\in\partial O\mid \dl(y)<r\}}.
\end{equation}
Denote $\bigcap_{r>0}\Sigma_r$ by $\Sigma$.
The following is the main theorem of this section.
\begin{theorem}\label{th:example}
Suppose that
\begin{equation}\label{eq:cap}
\Cp_{d-4}\left(\Sigma\right)=0.
\end{equation}
Then, $\bfone_{\Xi_{\overline{O}}}\in\BVloc\cap QR$.
In particular, the conclusion of \Thm{decomp} holds with $\rho=\bfone_{\Xi_{\overline{O}}}$.
Moreover, $\Cp^{\bfone_{\Xi_{\overline{O}}}}(\partial\Xi_{\overline{O}}\setminus\partial'\Xi_{\overline{O}})=0$.
\end{theorem}
\begin{remark}\label{rem:remark}
\begin{romanlist}[(v)]
\item The set $\Xi_{\overline{O}}$ coincides with the closure of $\Xi_O$ in $E$.
\item If $d\ge4$ and there exists a sequence of Borel subsets $\{B_n\}_{n=1}^\infty$ of $\R^d$ such that $\bigcup_{n=1}^\infty B_n=\R^d$ and $\fH^{d-4}(\Sg\cap B_n)<\infty$ for every $n\in\N$, then \Eq{cap} holds. 
\item Since $\inf_{z\in\partial O}\dl(z)>0$ implies $\Sigma=\emptyset$, the uniform exterior condition~\Eq{uebc} implies condition~\Eq{cap}.
\item The diffusion $\{X_t\}$ associated with $(\cE^\rho,\cF^\rho)$ for $\rho=\bfone_{\Xi_{\overline{O}}}$ is regarded as the modified reflecting Ornstein--Uhlenbeck process on $\Xi_{\overline{O}}$; the term ``modified'' is added since the domain $\cF^\rho$ is defined as the closure of smooth functions and it is not clear whether $\cF^\rho$ is the maximal domain.
\item The fact that $\Cp^{\bfone_{\Xi_{\overline{O}}}}(\partial\Xi_{\overline{O}}\setminus\partial'\Xi_{\overline{O}})=0$ implies that the measure $\|D\bfone_{\Xi_{\overline{O}}}\|$ concentrates on $\partial'\Xi_{\overline{O}}$, which means that the process $\{X_t\}$ reflects only at $\partial'\Xi_{\overline{O}}$ in view of \Eq{decomp}.
\end{romanlist}
\end{remark}
\begin{example}
Let $\ph\colon [1,2)\to\R_+$ be a convex function such that $\ph(1)=0$ and $\ph(t)>0$ for $t\in(1,2)$.
Suppose $d\ge4$ and consider an open subset $O$ of $\R^d$ that is defined by
\[
 O=B(0,2)\setminus\{(x_1,\dots,x_d)\in\R^d\mid x_d\in[1,2)\text{ and }\ph(x_d)^2\ge x_1^2+\dots+x_{d-1}^2\}.
\]
Then $\Sigma=\{(0,\dots,0,1)\}$ and $\fH^{d-4}(\Sigma)<\infty$. Therefore, $\bfone_{\Xi_{\overline{O}}}\in\BVloc$ from \Thm{example} and \Rem{remark}~(ii).
It is likely that $\bfone_{\Xi_{\overline{O}}}\notin BV$ in general, but I have no proof to offer at the moment.
\end{example}
The capacity $\Cp_{d-4}$ is involved in \Thm{example} for the following reason. From Theorem~1.1 of \cite{KS} and the arguments in Section~7 in that paper, for $a\in(0,T]$ and $M>0$, there exists a constant $C>0$ depending only on $a$ and $M$ such that, for any closed set $A$ in $\overline{B}(0,M)\subset \R^d$,
\begin{align*}
  C^{-1}\Cp_{d-4}(A)&\le \mathbf{P}(\{Z_s(t)\in A\text{ for some }(s,t)\in[a,T]\times[a,T]\})\\
  &\le C\Cp_{d-4}(A),
\end{align*}
where $(\{Z_s\}_{s\ge0},\mathbf{P})$ is the Ornstein--Uhlenbeck process on $E$ with initial distribution $\mu$.
This implies the following property.
\begin{equation}\label{eq:d4}
\parbox[l]{0.8\hsize}{%
 If a closed subset $A$ of $\R^d\setminus\{0\}$ satisfies $\Cp_{d-4}(A)=0$, then $\Cp(\Theta_A)=0$.}
\end{equation}
In particular, since $0\notin\Sg$, condition~\Eq{cap} implies 
\begin{equation}\label{eq:slim}
\Cp(\Theta_\Sg)=0.
\end{equation}

For the proof of \Thm{example}, we provide some quantitative estimates as in \cite{U,HU}.
We define a Lipschitz continuous function $q$ on $\R^d$ by
\[
  q(x)=\inf_{y\in O^c}|x-y|_{\R^d}-\inf_{y\in O}|x-y|_{\R^d},
 \quad x\in\R^d.
\]
For $r\ge0$, set $O_r=\{x\in\R^d\mid q(x)>r\}$.
Note that $O_0=O$ and $\{q(x)\ge0\}=\overline{O}$.

Let $W=C([0,\infty)\to\R^d)$ and $\cB(W)$ be the Borel $\sg$-field of $W$.
Let $\{\hat P_x\}_{x\in\R^d}$ be the probability measures on $W$ such that the coordinate process $\{\om_t\}_{t\ge0}$ is the $d$-dimensional Brownian motion starting at $x$ under $\hat P_x$ for each $x\in\R^d$.
For $t\ge0$, let $\hat\cF_t$ denote the $\sg$-field generated by $\{\om_s\mid s\in[0,t]\}$.
For an $\{\hat\cF_t\}$-stopping time $\tau$, define 
$\hat\cF_\tau$ as $\{A\in\cB(W)\mid A\cap\{\tau\le t\}\in\hat\cF_t \text{ for all }t\ge0\}$.
We denote the integral with respect to $\hat P_x$ by $\hat E_x$.
For $s>0$, the shift operator $\theta_s\colon W\to W$ is defined by $(\theta_s \om)_t=\om_{s+t}$, $t\ge0$.

The following claims (\Lem{ito}--\Prop{path}) are slight modifications of those in \cite{HU}.
\begin{lemma}\label{lem:ito}
Let $x\in\overline O$ and choose $y\in\partial O$ such that $q(x)=|x-y|_{\R^d}$.
Suppose that $\dl(y)>0$.
Let  $\dl\in (0,\dl(y))$ and take $z\in O^c$ such that $B(z,\dl)\cap\overline O=\{y\}$.
Let $C_\dl=(d-1)/(2\dl)$ and $R_t=|\om_t-z|_{\R^d}$ for $\om=\{\om_t\}\in W$.
Then, for each $u>0$,
\[
\{R_t\ge\dl \text{ for all }t\in[0,u]\}\subset\{R_t\le q(x)+\dl+C_\dl t+S_t \text{ for all }t\in[0,u]\}
\]
up to a $\hat P_x$-null set.
Here, $S_t$ is the 1-dimensional Brownian motion under $\hat P_x$ starting at $0$ that is defined by
\[
   S_t(\om)=\sum_{i=1}^d\int_0^t\frac{\om_s^{(i)}-z^{(i)}}{R_s}\,d\om^{(i)}_s,\quad
   \om_s=(\om^{(1)}_s,\ldots,\om^{(d)}_s),\ 
   z=(z^{(1)},\ldots,z^{(d)}),
\]
up to the $\{\hat\cF_t\}$-stopping time $\inf\{t\ge0\mid R_t=0\}$.
\end{lemma}
\begin{proof}
Define $\sg=\inf\{t\ge0\mid R_t=0\}$.
Note that $R_0=|x-z|_{\R^d}=q(x)+\dl$ $\hat P_x$-a.e.
By virtue of It\^o's formula,
\[
  R_t=q(x)+\dl+\int_0^t\frac{d-1}{2R_s}\,ds+S_t
  \quad\text{on }\{t<\sg\}\quad
  \hat P_x\text{-a.e.}
\]
Therefore, the assertion holds.
\end{proof}
\begin{proposition}\label{prop:1}
In the same situation as in \Lem{ito}, for every $u>0$,
\[
  \hat P_x\left[\inf_{t\in[0, u]}q(\om_t)\ge0\right]\le \left(\frac{d-1}{\dl(y)}+u^{-1/2}\right)q(x).
\]
\end{proposition}
\begin{proof}
Take $\dl\in(0,\dl(y))$.
From Lemma~\ref{lem:ito}, 
\[
\hat P_x\left[ R_t\ge\dl\text{ for all }t\in[0,u]\right]
\le \hat P_x\left[q(x)+\dl+C_\dl t+S_t\ge\dl\text{ for all }t\in[0,u]\right].
\]
Let $r>q(x)$ and define $\eta=\inf\{t\ge0\mid C_\dl t+S_t\le -r\}$.
The law of $\eta$ under $\hat P_x$ is given by 
\[
\hat P_x[\eta\in dt]=\bfone_{(0,\infty)}(t)\frac{r}{\sqrt{2\pi t^3}}\exp\left(-\frac{(r+C_\dl t)^2}{2t}\right)\,dt+(1-e^{-2C_\dl r})\dl_{\infty}(dt),
\]
where $\dl_\infty$ is the delta measure at $\infty$ (see, e.g., \cite[p.~295]{BS}).
Then, we have
\begin{align*}
\hat P_x\left[\inf_{t\in[0, u]}q(\om_t)\ge0\right]
&\le \hat P_x\left[ R_t\ge\dl\text{ for all }t\in[0,u]\right]
\le \hat P_x[\eta>u]\\
&=\int_u^\infty \frac{r}{\sqrt{2\pi t^3}}\exp\left(-\frac{(r+C_\dl t)^2}{2t}\right)\,dt
+1-e^{-2C_\dl r}\\
&\le \int_u^\infty \frac{r}{\sqrt{2\pi t^3}}\,dt
+2C_\dl r
=\sqrt{\frac2\pi}\frac{r}{\sqrt u}+2C_\dl r.
\end{align*}
Letting $r\to q(x)$ and $\dl\to \dl(y)$, we obtain the claim.
\end{proof}
For $r>0$, define an $\{\hat\cF_t\}$-stopping time $\tau_r$ by $\tau_r=\inf\{t\ge0\mid \om_t\not\in O_r\}$.
Let $\hat P_x^r$ be the law of $\tau_r$ under $\hat P_x$.
The following lemma is the same as \cite[Lemma~3.2]{HU}, so the proof is omitted.
\begin{lemma}
\label{lem:differentiable}
$\hat P_x^r([0,t])$ is differentiable in $t$ on $(0,\infty)$ and there exists a constant $\hat C_1>0$ independent of $x$, $r$ and $t$ such that
$
\frac{d}{dt}\hat P_x^r([0,t])\le \hat C_1 t^{-1}
$.
\end{lemma}
For a closed subset $A$ of $\R^d$, a stopping time $\sg_A$ is defined as $\sg_A=\inf\{t\ge0\mid \om_t\in A\}$.
For $r>0$, let 
\begin{equation}\label{eq:Agm}
  A_r=\{x\in\R^d\mid \overline{B}(x,r)\cap \Sigma_r\ne\emptyset\},
\end{equation}
where $\Sg_r$ is defined in \Eq{Sigma}.
\begin{proposition}\label{prop:path}
For $\gm>0$, $u>0$, $r\in(0,\gm]$, and  $x\in \overline{O}$, the inequality
\begin{equation}\label{eq:prob}
  \hat P_x\left[0\le\inf_{t\in[0,u]}q(\om_t)\le r,\ \sg_{A_\gm}>u\right]\le\left(\frac{d-1}\gm+ \hat C_2u^{-1/2}\right)r
\end{equation}
holds with $\hat C_2=4\hat C_1+2$.
\end{proposition}
\begin{proof}
Denote $(d-1)/\gm$ by $C'_\gm$.
If $x\in A_\gm$, \Eq{prob} is trivial since the left-hand side is $0$.
Suppose $x\in \overline{O}\setminus (O_r\cup A_\gm)$. Then, $q(x)<r\le\gm$.
Take $y\in \partial O$ such that $|x-y|_{\R^d}=q(x)$.
Because $x\notin A_\gm$ and $q(x)<\gm$, $y$ does not belong to $\Sigma_\gm$: that is, $\dl(y)\ge \gm$.
From \Prop{1}, 
\[
\hat P_x\left[0\le\inf_{t\in[0,u]}q(\om_t)\le r\right]
\le \left(\frac{d-1}{\dl(y)}+u^{-1/2}\right)q(x)
\le (C'_\gm+u^{-1/2})r.
\]
Next, suppose $x\in O_r\setminus A_\gm$. Then,
\begin{align*}
&\left\{0\le\inf_{t\in[0,u]}q(\om_t)\le r\right\}
= \left\{\tau_r\le u,\ 0\le\inf_{t\in[0,u-\tau_r]}q((\theta_{\tau_r}\om)_t)\right\}\\
&\subset \bigcup_{k=1}^\infty \left\{2^{-k}u<u-\tau_r\le 2^{-k+1}u,\ 0\le\inf_{t\in[0,2^{-k}u]}q((\theta_{\tau_r}\om)_t)\right\}\cup \{\tau_r=u\}.
\end{align*}
From Lemma~\ref{lem:differentiable}, $\hat P_x[\tau_r=u]=0$.
From the strong Markov property and \Prop{1},
\begin{align*}
&\hat P_x\left[2^{-k}u<u-\tau_r\le 2^{-k+1}u,\ 0\le\inf_{t\in[0,2^{-k}u]}q((\theta_{\tau_r}\om)_t),\ \sg_{A_\gm}>u\;\vrule\;\hat\cF_{\tau_r}\right]\\
&\le \bfone_{\{2^{-k}u<u-\tau_r\le 2^{-k+1}u\}}\cdot \hat E_x\left[\hat P_{\om_{\tau_r}}\!\left[0\le\inf_{t\in[0,2^{-k}u]}q(\om_t)\right];\om_{\tau_r}\in \partial O_r\setminus A_\gm\right]\\
&\le \bfone_{\{2^{-k}u<u-\tau_r\le 2^{-k+1}u\}}\cdot (C'_\gm+(2^{-k}u)^{-1/2})r\\
&\le \bfone_{\{2^{-k}u<u-\tau_r\le 2^{-k+1}u\}}\cdot (C'_\gm+((u-\tau_r)/2)^{-1/2})r.
\end{align*}
Therefore, 
\begin{align*}
&\hat P_x\left[0\le\inf_{t\in[0,u]}q(\om_t)\le r,\ \sg_{A_\gm}>u\right]\\
&\le r\hat E_x[C'_\gm+((u-\tau_r)/2)^{-1/2};\ \tau_r\le u]\\
&\le rC'_\gm+r\left(\frac{u}{4}\right)^{-1/2}\hat P_x[\tau_r\le u/2]
+r\int_{u/2}^u \left(\frac{u-s}2\right)^{-1/2}\hat C_1s^{-1}\,ds\\*
&\qquad\text{(from \Lem{differentiable})}\\
&\le rC'_\gm+2ru^{-1/2}+\frac{2\hat C_1 r}{u}\int_{u/2}^u\left(\frac{u-s}2\right)^{-1/2}\,ds\\
&=rC'_\gm+2ru^{-1/2}+4\hat C_1r u^{-1/2}.
\end{align*}
This completes the proof.
\end{proof}
\begin{lemma}\label{lem:dec}
Let $\{B_n\}_{n=1}^\infty$ be a decreasing sequence of closed subsets of $\R^d$.
If $\Cp_{d-4}\left(\bigcap_{n=1}^\infty B_n\right)=0$, then $\Cp(\Theta_{B_n})$ converges to $0$ as $n\to\infty$.
\end{lemma}
\begin{proof}
From \Eq{d4}, $\Cp(\Theta_{\bigcap_{n=1}^\infty B_n})=0$.
For $\eps>0$, take a compact subset $F$ of $E$ such that $\Cp(E\setminus F)<\eps$.
Since $\{\Theta_{B_n}\cap F\}_{n=1}^\infty$ is a decreasing sequence of compact sets, 
\[
\lim_{n\to\infty}\Cp(\Theta_{B_n}\cap F)=\Cp(\Theta_{\bigcap_{n=1}^\infty B_n}\cap F)=0.
\]
Therefore, 
\[
\lim_{n\to\infty}\Cp(\Theta_{B_n})\le\lim_{n\to\infty}\Cp(\Theta_{B_n}\cap F)+\Cp(E\setminus F)\le \eps.
\]
Since $\eps>0$ is arbitrary, we obtain the claim.
\end{proof}

\begin{proof}[Proof of \Thm{example}]
Recall the definition of  $A_r$ in \Eq{Agm}. We note that
\begin{equation}\label{eq:A}
\bigcap_{k=1}^\infty A_{1/k}=\Sg.
\end{equation}
From \Prop{path} and \Eq{slim},
\begin{align*}
\mu(\Xi_{\overline{O}}\setminus \Xi_{O})
&=\hat P_0\left[\inf_{t\in[0,T]}q(\om_t)=0\right]\\
&\le \lim_{n\to\infty}\hat P_0\left[\left\{\inf_{t\in[0,T]}q(\om_t)=0\right\}\cap\{\sg_{A_{1/n}}>T\}\right]
+\mu(\Theta_\Sg)\\
&=0.
\end{align*}
That is, $\bfone_{\Xi_{\overline{O}}}=\bfone_{\Xi_{O}}$ $\mu$-a.e.
Since $\bfone_{\Xi_{O}}$ is a lower semicontinuous function, $\bfone_{\Xi_{\overline{O}}}\in QR$ (cf.\ \cite[p.~230]{Fu}).

For $k\in\N$, let $B_k$ denote the closure of $A_{1/k}^c$, and let $F_k=\Xi_{B_k}$.
Then, $\{F_k\}_{k=1}^\infty$ is an $\cE$-nest from \Eq{slim}, \Eq{A}, and \Lem{dec}.
Since $A_{1/(k+1)}\subset B_k^c$, $\Theta_{A_{1/(k+1)}}\subset E\setminus F_k$ for each $k$.
Let $h(w)=\inf_{t\in[0,T]}q(w(t))$ for $w\in E$ and $\chi(t)=(0\vee t)\wg1$ for $t\in\R$.
For $m\in\N$, let $f_m(w)=\chi(m h(w))$ for $w\in E$.
Then, $f_m\in \D^{1,1}$ and $f_m$ converges to $\bfone_{\Xi_{O}}$ in $L^1$ as $m\to\infty$.
Moreover, for $k\in\N$ and $m\ge k+1$,
\begin{align*}
\|(\nabla f_m)\cdot \bfone_{F_k}\|_1
&\le \int_{F_k} m\cdot \bfone_{\{0\le h\le 1/m\}}\,d\mu\\
&\le m \mu\left(\{0\le h\le 1/m\}\setminus \Theta_{A_{1/(k+1)}}\right)\\
&\le (k+1)(d-1)+\hat C_2 T^{-1/2}
\quad\text{(from \Prop{path})}.
\end{align*}
Therefore, $\sup_{m\in\N}\|(\nabla f_m)\cdot\bfone_{F_k}\|_1<\infty$ for each $k\in\N$.
From \Thm{1}, we conclude that $\bfone_{\Xi_{\overline{O}}}\in \BVloc$.

The identity $\Cp^{\bfone_{\Xi_{\overline{O}}}}(\partial\Xi_{\overline{O}}\setminus\partial'\Xi_{\overline{O}})=0$ is proved in a similar way to the proof of \cite[Theorem~2.4]{HU}.
We provide details for the readers' convenience.
Take an $\cE$-nest $\{F'_k\}_{k=1}^\infty$ and functions $\{\ph_k\}_{k=1}^\infty$ obtained by applying \Lem{nest} to the $\cE$-nest $\{F_k\}_{k=1}^\infty$ defined above.
For $s\in(0,T)$, we define a closed set $V_{s}$ by
\[
 V_{s}=\left\{w\in E\;\vrule\; \inf_{t\in[0,s]}q(w(t))=0\text{ and } \inf_{t\in[s,T]}q(w(t))=0\right\}.
\]
Then, $\partial\Xi_{\overline{O}}\setminus\partial'\Xi_{\overline{O}}\subset \bigcup_{s\in(0,T)\cap \Q}V_{s}$.
Fix $s\in(0,T)$ and define a map $f\colon E\to\R^2$ by
\[
f(w)=\Bigl(\inf_{t\in[0,s]}q(w(t)),\ \inf_{t\in[s,T]}q(w(t))\Bigr).
\]
Fix $k\in\N$ and take $l\in\N$ such that $\ph_k\in \D^{1,2}_{F_l}$.
Let $\eps>0$.
Henceforth, $C$ denotes an unimportant positive constant independent of $\eps$ that may vary line by line.
Take a smooth function $g$ on $[0,\infty)$ such that
\[
g(t)=\begin{cases}
 1& t\in[0,e^{-2/\eps}],\\
 {-3\eps\log t}-4 & t\in[e^{-14/(9\eps)},e^{-13/(9\eps)}],\\
 0 & t\in[e^{-1/\eps},\infty),
\end{cases}
\]
and $-3\eps/t\le g'(t)\le0$ for all $t>0$.
We define a function $\zeta\colon {\R}^2\to\R$ by $\zeta(x,y)=\sqrt{x^2+y^2}$ and set
$\iota = g\circ \zeta $.
Since $\iota\circ f$ is a bounded $H$-Lipschitz continuous function, it belongs to ${\D}^{1,2}_b$.
Let $\eta=(\iota\circ f)\ph_k+e_{E\setminus F'_k}\in \D^{1,2}$.
Then, $\eta\ge1$ $\mu$-a.e.\ on some open set including $V_{s}$.
Since $(\iota\circ f)\ph_k\in{\D}^{1,2}_{F_{l}}$,
\begin{align*}
&\cE^{\bfone_{\Xi_{\overline{O}}}}(\eta|_{\Xi_{\overline{O}}},\eta|_{\Xi_{\overline{O}}})\\
&\le 2\cE^{\bfone_{\Xi_{\overline{O}}}}(((\iota\circ f)\ph_k)|_{\Xi_{\overline{O}}},((\iota\circ f)\ph_k)|_{\Xi_{\overline{O}}})+ 2\cE^{\bfone_{\Xi_{\overline{O}}}}(e_{E\setminus F'_k}|_{\Xi_{\overline{O}}},e_{E\setminus F'_k}|_{\Xi_{\overline{O}}})\\
&\le 2\int_{\Xi_{\overline{O}}\cap F_{l}} |\nabla (\iota\circ f)|^2_H \, d\mu+4\cE(\ph_k,\ph_k)+2\cE(e_{E\setminus F'_k},e_{E\setminus F'_k})
.
\end{align*}
Denoting the gradient operator on $\R^2$ by $\nabla_{\R^2}$, we have
\begin{align*}
\int_{\Xi_{\overline{O}}\cap F_{l}} |\nabla (\iota\circ f)|^2_H \, d\mu
&=\int_{\Xi_{\overline{O}\cap B_{l}}} |\langle (\nabla f)(w),(\nabla_{{\R}^2} \,\iota)(f(w)) \rangle_{{\R}^2}|^2_{H} \, \mu(dw)\\
&\le C\int_{\Xi_{\overline{O}\cap B_{l}}}  |(\nabla_{\R^2} \,\iota)(f(w))|_{\R^2}^2\, \mu(dw)\\
&=C\int_{\{(x,y)\in \R^2 \mid x\geq 0 ,y\geq 0\}}|\nabla_{\R^2} \,\iota|_{\R^2}^2\, d(f_*(\mu|_{\Xi_{\overline{O}\cap B_{l}}}))\\
&=:I_1.
\end{align*}
In the first line, $\langle \cdot,\cdot \rangle_{{\R}^2}$ denotes a pairing between elements in $H\otimes \R^2$ and in $\R^2$ that takes values in $H$.
We note that
\begin{align*}
|\nabla_{\R^2} \,\iota|_{\R^2}^2
&=\left({\partial \iota}/{\partial x}\right)^2+\left({\partial \iota}/{\partial y}\right)^2	\\
&=\left(g'\circ \zeta(x,y)\right)^2\frac{x^2}{x^2+y^2}+\left(g'\circ \zeta(x,y)\right)^2\frac{y^2}{x^2+y^2}\\
&=\left(g'\circ \zeta(x,y)\right)^2.
\end{align*}
By letting $\psi=(\zeta\circ f)_*(\mu|_{\Xi_{\overline{O}\cap B_{l}}})$, we obtain 
\[
	I_1=C\int_{0}^{\infty}g'(r)^2\,\psi(dr) \le 9C\eps^2\int_{e^{-2/\eps}}^{e^{-1/\eps}}r^{-2} \,\psi(dr)=:I_2.
\]
We now have
\begin{align*}
\Psi(r )&:=\psi([0,r])
= (f_*(\mu|_{\Xi_{\overline{O}\cap B_{l}}}))(\zeta^{-1}([0,r]))\\
&=\mu\left[\left\{ w\in \Xi_{\overline{O}\cap B_{l}} \;\vrule\; \inf_{t\in[0,s]}q(w(t))^2+\inf_{t\in[s,T]}q(w(t))^2 \le r^2\right\} \right]\\
&\le\mu\left[\left\{ w\in \Xi_{\overline{O}\setminus A_{1/(l+1)}} \;\vrule\; \begin{array}{l}0\le\inf_{t\in[0,s]}q(w(t))\le r,\\ 0\le\inf_{t\in[s,T]}q(w(t)) \le r\end{array}\right\} \right]\\
&\le\hat E_0\biggl[\hat P_{\om_{s}}\left[ 0\le \inf_{t\in[0,T-s]}q(\om_t)\le r,\ \sg_{A_{1/(l+1)}}>T-s  \right];\\*
&\qquad\qquad 0\le \inf_{t\in[0,s]}q(\om_t)\le r,\ \sg_{A_{1/(l+1)}}>s\biggr].
\end{align*}
From \Prop{path}, if $r\le1/(l+1)$, this expectation is dominated by
\begin{align*}
&((d-1)(l+1)+C(T-s)^{-1/2})r\hat P_0\left[ 0 \le \inf_{t\in[0,s]}q(\om_t)\le r,\ \sg_{A_{1/(l+1)}}>s  \right]\\
&\le  ((d-1)(l+1)+C(T-s)^{-1/2})((d-1)(l+1)+Cs^{-1/2})r^2\\
&=Cr^2.	
\end{align*}
Thus, for $\eps\le 1/\log(l+1)$,
\begin{align*}
I_2&=9C\eps^2\int_{e^{-2/\eps}}^{e^{-1/\eps}}r^{-2} \,d\Psi(r)\\
&=C\eps^2 \left\{\left[\frac{\Psi(r)}{r^2}\right]_{e^{-2/\eps}}^{e^{-1/\eps}}+\int_{e^{-2/\eps}}^{e^{-1/\eps}}\frac{2\Psi(r)}{r^3} \,dr\right\}\\
&\le C\eps^2  \Bigl(C+\int_{e^{-2/\eps}}^{e^{-1/\eps}}\frac{C}{r}  \,dr\Bigr)
=C\eps^2 (1+1/\eps).
\end{align*}
Therefore,
\begin{align*}
\Cp^{\bfone_{\Xi_{\overline{O}}}}(V_{s})
&\le \cE^{\bfone_{\Xi_{\overline{O}}}}(\eta|_{\Xi_{\overline{O}}},\eta|_{\Xi_{\overline{O}}}) + \|\eta|_{\Xi_{\overline{O}}}\|^2_{L^2(\mu|_{\Xi_{\overline{O}}})}\\
&\le 2\int_{\Xi_{\overline{O}}\cap F_{l}} |\nabla (\iota\circ f)|^2_H \, d\mu+4\cE(\ph_k,\ph_k)+2\cE(e_{E\setminus F'_k},e_{E\setminus F'_k})\\&\quad+2\Psi(e^{-1/\eps})+2\|e_{E\setminus F'_k}\|_2^2\\
&\le C\eps^2(1+1/\eps)+4\cE(\ph_k,\ph_k)+2\Cp(E\setminus F'_k)+Ce^{-2/\eps}.
\end{align*}
By letting $\eps \to 0$ and then $k\to\infty$, we obtain $\Cp^{\bfone_{\Xi_{\overline{O}}}}(V_{s})=0$.
Therefore, 
\[
\Cp^{\bfone_{\Xi_{\overline{O}}}}(\partial \Xi_{\overline{O}}\setminus \partial' \Xi_{\overline{O}})
\le\sum_{s\in(0,T)\cap\Q}\Cp^{\bfone_{\Xi_{\overline{O}}}}(V_{s})
=0.\tag*{\popQED}
\]
\end{proof}
\begin{remark}
In this section, we considered only one-sided pinned path spaces as the underlying space for simplicity.
However, the general idea of the argument is also valid for pinned path spaces, as discussed in \cite{HU}.
\end{remark}
\section*{Acknowledgements}
This research was partially supported by JSPS KAKENHI Grant Number 24540170.

\printindex                         
\end{document}